\documentclass{article}

\usepackage[T1]{fontenc}
\usepackage[english]{babel}
\usepackage{amsmath}
\usepackage{amsthm}
\usepackage{amssymb}

\newtheorem{thm}{Theorem}

\newtheorem{lem}{Lemma}
\newtheorem{cor}{Corollary}

\theoremstyle{definition}

\theoremstyle{remark}
\newtheorem{rem}{Remark}

\theoremstyle{remark}
\newtheorem{ex}{Example}

\def\cM{\mathcal M}

\def\cP{\mathcal P}
\def\cC{\mathcal C}

\def\cH{\mathcal H}
\def\cV{\mathcal V}
\def\cW{\mathcal W}

\def\bR{\mathbb R} 

\def\bS{\mathbb S}
\def\bN{\mathbb N}

\def\xx{\mathbf x}

\def\Ker{\operatorname{Ker}}

\def\mD{\tilde\Delta}

\def\pa{\partial}

\def\bul{\;\bullet}
\def\wed{\;\wedge}

\def\sl{\mathfrak{sl}}
\def\osp{\mathfrak{osp}}

\begin{document}

\title{The Fischer Decomposition for the $H$-action and Its Applications}

\author{Roman L\' avi\v cka\thanks{I acknowledge the financial support from the grant GA 201/08/0397.
This work is also a part of the research plan MSM 0021620839, which is financed by the Ministry of Education of the   Czech Republic.}}

\date{}

\maketitle

\begin{abstract}
Recently the Fischer decomposition  for the $H$-action  of the Pin group on Clifford algebra valued polynomials has been obtained. We apply this tool to get various decompositions of special monogenic and inframonogenic polynomials in terms of two sided monogenic ones. 

\medskip\noindent
{\bf Keywords:} Fischer decomposition, special monogenic polynomials, inframonogenic polynomials, two sided monogenic polynomials

\medskip\noindent
{\bf AMS classification:} Primary 30G35; Secondary 58A10
\end{abstract}


\section{Introduction}
The classical Fischer decomposition of Clifford algebra valued (or spinor valued) polynomials is one of the most important fact concerning solutions of the Dirac equation in the Euclidean space $\bR^m.$ It leads immediately
to the right form of the Taylor (and Laurent) series for monogenic functions and has further important applications in Clifford analysis (see, e.g., \cite{BDS2,DSS}).  

A lot of attention has been recently devoted to study of properties of various classes of special solutions of the Dirac equation. For example, special solutions might be just those with values restricted to a~given subspace $V$ of the Clifford algebra. The case when the space $V$ is, in addition, invariant under an action of the Pin group seems to be the most interesting. In particular, if we consider the space $V=\bR_{0,m}^s$ of $s$-vectors in the Clifford algebra $\bR_{0,m}$ and the so-called $H$-action on the space $\cC^\infty(\bR^m,V)$ of smooth functions in $\bR^m$ taking  values in the subspace $V,$ then the Dirac equation reduces to
the Hodge-de Rham system. 
Similarly, for other suitable invariant subspaces of values, we get the so-called generalized Moisil-Th\'eodoresco systems. 

These systems have been recently  carefully studied and the corresponding
Fischer decomposition and the analogues of the Howe dual pairs have been described in a series of
papers \cite{R1,R2,DLS,DLS3}. It is well-known that the monogenic Fischer decomposition  is a refinement of the classical Fischer decomposition for scalar valued functions. It was shown recently that the Fischer decomposition for monogenic functions under the $H$-action can be viewed even as a refinement of the monogenic Fischer decomposition for Clifford algebra valued functions (see \cite{BSES,DLS2}).

The main aim of the paper is to describe an application of the decompositions mentioned above to
a study of the Fischer decomposition for inframonogenic functions introduced recently
by H. Malonek, D. Pe\~na Pe\~na and F. Sommen in papers  \cite{MPS1,MPS2}. In Section 2, we give a short review
of the Fischer decomposition for the $H$-action. In Section 3, we summarize the Fischer decomposition for special monogenic functions, that is, those taking values in a~given $H$-invariant subspace $V$ of the Clifford algebra $\bR_{0,m}.$ In particular, we describe quite explicitly the Fischer decomposition for solutions of generalized Moisil-Th\'eodoresco systems, cf. \cite{DLS}. In Section 4, we study spaces of homogeneous inframonogenic polynomials and we obtain their Fischer decomposition (see Theorem \ref{tdecompi}). 

As is explained in detail in \cite{BDS}, these results can be easily translated into the language of differential forms.

\section{The Fischer Decomposition for the $H$-action }

In this section, we describe the Fischer decomposition for the $H$-action  obtained recently in \cite{DLS2} using results from \cite{hom}. 

Before doing so we recall the well-known decomposition of spinor valued polynomials.
Let us denote by $\bS$ a~basic spinor representation for the Pin group $Pin(m)$ of the Euclidean space $\bR^m$ and by
$\cP(\bS)$ the space of $\bS$-valued polynomials in the vector variable $\xx=(x_1,\ldots,x_m)$ of $\bR^m.$
On the space $\cP(\bS),$ we can consider the so-called $L$-action of $Pin(m)$ given by
\begin{equation}\label{Laction}
[L(r)(P)](\xx)=r\,P(r^{-1}\xx\; r),\ r\in Pin(m),\ P\in\cP(\bS)\text{\ \ and\ \ }\xx\in\bR^m.
\end{equation}
Let the vectors $e_1,\ldots,e_m$ form the standard basis of $\bR^m.$ Then it is easily seen that the multiplication by the vector variable $\xx=e_1x_1+\cdots+e_mx_m$
and the Dirac operator $$\pa =e_1\pa_{x_1}+\cdots+e_m\pa_{x_m}$$ (both applied from the left)
are examples of invariant linear operators on the space $\cP(\bS)$ with the $L$-action.
Actually, the invariant operators $\xx$ and $\pa$ are basic in the sense that they generate the Lie superalgebra $\osp(1|2)$ which gives the hidden symmetry of the space $\cP(\bS),$ see \cite{DLS3} for details.
Using this observation,  the Fischer decomposition for  spinor valued polynomials is given in \cite{BSES}
as a~special case of the general theory of the Howe duality (see \cite{how}).
Indeed, denote by $\cM_k(\bS)$ the space of $k$-homogeneous polynomials $P\in\cP(\bS)$ which are left monogenic,
that is,  those satisfying the Dirac equation $$\pa P=0.$$
Then in this case the Fischer decomposition reads as follows:
\begin{equation}\label{FischerPin}
\cP(\bS)=\bigoplus_{k=0}^{\infty}\bigoplus_{p=0}^{\infty}\xx^p\cM_k(\bS).
\end{equation}
In addition, we know that, under the $L$-action, all submodules $\cM_k(\bS)$ are irreducible and mutually inequivalent. Recall that the decomposition (\ref{FischerPin}) follows easily from the classical Fischer decomposition of the space $\cP$ of scalar valued polynomials in $\bR^m$ which is given by
\begin{equation}\label{harmFischer}
\cP=\bigoplus_{k=0}^{\infty}\bigoplus_{p=0}^{\infty}|\xx|^{2p}\Ker_k\Delta.
\end{equation}
Here $|\xx|^2=x_1^2+\cdots+x_m^2=-\xx^2,$ $\Delta=\pa^2_{x_1}+\cdots+\pa^2_{x_m}=-\pa^2$ is the Laplace operator
and $\Ker_k\Delta$ is the space of $k$-homogeneous polynomials $P\in\cP$ such that $\Delta P=0.$
Indeed, we get easily the decomposition (\ref{FischerPin}) from (\ref{harmFischer}) using the fact that
\begin{equation}\label{harmmonog}
(\Ker_k\Delta)\otimes\bS=\cM_k(\bS)\oplus\xx\cM_{k-1}(\bS).
\end{equation}

Now we are going to deal with the case of Clifford algebra valued polynomials.
Let $\bR_{0,m}$ be the real Clifford algebra over $\bR^m$ satisfying the relations
$$e_ie_j+e_je_i=-2\delta_{ij}.$$
The Clifford algebra $\bR_{0,m}$ can be viewed naturally as the graded associative algebra
$$\bR_{0,m}=\bigoplus_{s=0}^m\bR_{0,m}^s$$
where $\bR_{0,m}^s$ denotes the space of $s$-vectors in $\bR_{0,m}.$
As usual we identify $\bR_{0,m}^1$ with $\bR^m.$
For an 1-vector $u$ and an $s$-vector $v,$ the Clifford product $uv$ splits into the sum of an $(s-1)$-vector $u\bullet v$ and an $(s+1)$-vector $u\wedge v.$ Indeed, we have that
$$uv=u\bullet v+u\wedge v\text{\ \ \ with\ \ \ }u\bullet v=\frac 12(uv-(-1)^svu)\text{\ \ and\ \ }u\wedge v=\frac 12(uv+(-1)^svu).$$
By linearity, we extend the so-called inner product $u\bullet v$ and the outer product $u\wedge v$ for an 1-vector $u$ and an arbitrary Clifford number $v\in\bR_{0,m}.$

In what follows, we deal with the space $\cP^*$ of $\bR_{0,m}$-valued polynomials in $\bR^m.$
Each polynomial $P\in\cP^*$ is of the form $$P(\xx)=\sum_{\alpha} a_{\alpha}\xx^{\alpha},\ x\in\bR^m$$ where the sum is taken over a~finite subset of multiindeces $\alpha=(\alpha_1,\ldots,\alpha_m)$ of $\bN_0^m,$
all coefficients $a_{\alpha}$ belong to $\bR_{0,m}$ and $\xx^{\alpha}=x_1^{\alpha_1}\cdots x_m^{\alpha_m}.$
Denote by $\cP^*_k$ the space of $k$-homogeneous polynomials of $\cP^*$ and by $\cP^s_k$ the space of $s$-vector valued polynomials of $\cP^*_k.$
In general, for $\cV\subset\cP^*$ put
$\cV_k=\cV\cap\cP^*_k$ and $\cV^s_k=\cV\cap\cP^s_k.$

As we have mentioned above, on the space $\cP(\bS)$ with the $L$-action,
$\pa$ and $\xx$ are basic invariant operators.
On the space $\cP^*$ of Clifford algebra valued polynomials we can consider yet another action, namely,
the so-called $H$-action given by
\begin{equation}\label{Haction}
[H(r)P](\xx)=r\;P(r^{-1}\xx\; r)\;r^{-1},\ r\in Pin(m),\ P\in\cP^*\text{\ \ and\ \ }\xx\in\bR^m.
\end{equation}
Then a~natural question arises
which invariant operators are basic under the $H$-action.
It is not difficult to see that,
under the $H$-action, the operators $\pa$ and $\xx$ are still invariant.
On the other hand, we can split the left multiplication by 1-vector $\xx$ into the outer multiplication $\xx\wed$ and the inner multiplication $\xx\bul,$ that is, $$\xx=\xx\wedge+\;\xx\bullet .$$
Analogously,
the Dirac operator $\pa$ can be split also into two parts
$\pa=\pa^++\pa^-$ where
$$\pa^+P=\sum_{j=1}^m e_j\wedge(\pa_{x_j}P)\text{\ \ and\ \ }\pa^-P=\sum_{j=1}^m e_j\bullet(\pa_{x_j}P).$$

Actually,
the operators $\pa^+,$ $\pa^-,$ $\xx\wed$ and $\xx\bul$ are basic invariant operators for the $H$-action.
Indeed, these operators generate the Lie superalgebra $\sl(2|1)$ which gives the hidden symmetry of the space $\cP^*,$ see \cite{DLS3} for details. As suggested by the general theory of the Howe duality, the corresponding Fischer decomposition can be obtained even in this case (see \cite{DLS2}).
Recall that the space of spinor valued polynomials decomposes into the direct sum of multiples of spaces of homogeneous solutions of the Dirac equation by non-negative integer powers of $\xx,$ see (\ref{FischerPin}).
It turns out that, for the $H$-action, basic building blocks are spaces of 'homogeneous' solutions of the Hodge-de Rham system of equations
\begin{equation}\label{HdR}
\pa^+P=0,\ \pa^-P=0
\end{equation}
and their multiples by non-trivial words in the letters $\xx\wed$ and $\xx\bullet.$
Note that $(\pa^+)^2=0,$ $(\pa^-)^2=0,$ $\xx\wedge\xx\wed=0$ and $\xx\bullet\xx\bul=0.$ In particular, we have that the set $\Omega$ of all non-trivial words in the letters $\xx\wed$ and $\xx\bullet$
looks like
\begin{equation}\label{Omega}
\Omega=\{1,\ \xx\wed,\ \xx\bul,\ \xx\wedge\xx\bul,\ \xx\bullet\xx\wed,\ \xx\wedge\xx\bullet\xx\wed,\ \xx\bullet\xx\wedge\xx\bul, \ldots\}.
\end{equation}
Let $\cH$ be the space of $\bR_{0,m}$-valued polynomials $P$ in $\bR^m$ satisfying the Hodge-de Rham system (\ref{HdR}). Recall that
$\cH^s_k$ is then the subset of $s$-vector valued polynomials $P$ of $\cH$ which are homogeneous of degree $k.$
The Fischer decomposition for the $H$-action reads as follows (see \cite{DLS2}).

\begin{thm}\label{tFischerH}
The space $\cP^*$ of $\bR_{0,m}$-valued polynomials in $\bR^m$ decomposes as
\begin{equation}\label{FischerH}
\cP^*=\bigoplus_{s=0}^{m}\bigoplus_{k=0}^{\infty}\bigoplus_{w\in\Omega}w\cH^s_k.
\end{equation}
\end{thm}

\begin{rem}
(i) In addition, we have that $\cH^s_k=\{0\}$ just for $s\in\{0,m\}$ and $k\geq 1,$ 
$\cH^0_0=\bR$ and $\cH^m_0=\bR e_1e_2\cdots e_m.$
Moreover,
under the $H$-action, all non-trivial modules $\cH^s_k$ are irreducible and mutually inequivalent.

\medskip\noindent
(ii) It is easy to see that 
$w\cH^s_k=\{0\}$ if either $s=0$ and the word $w$ begins with the letter $\xx\bul$ or
$s=m$ and the word $w$ begins with the letter $\xx\wed.$ Otherwise, $$w\cH^s_k\simeq\cH^s_k.$$
\end{rem}

Now we sketch a~proof of Theorem \ref{tFischerH}, see \cite{DLS2} for details.
By classical Fischer decomposition (\ref{harmFischer}), it is easy to see that
\begin{equation}\label{hFischer}
\cP^*=\bigoplus_{s=0}^{m}\bigoplus_{k=0}^{\infty}\bigoplus_{p=0}^{\infty}|\xx|^{2p}\Ker^s_k\Delta.
\end{equation}
where $\Ker^s_k\Delta=\{P\in\cP_k^s: \Delta P=0\}.$
Realizing that $\Delta=-(\pa^+\pa^-+\;\pa^-\pa^+)$
and $|\xx|^2=-(\xx\wedge\xx\bullet+\;\xx\bullet\xx\wed),$
the decomposition (\ref{FischerH}) then follows easily from the next result obtained by Y. Homma in \cite{hom}.

\begin{thm}\label{hom} Let $0\leq s\leq m$ and $k\in\bN_0.$ Then,
under the $H$-action, $\Ker^s_k\Delta$ decomposes into inequivalent irreducible pieces as
$$
\Ker^s_k\Delta=\cH^s_k\oplus\xx\wedge \cH^{s-1}_{k-1}\oplus\xx\bullet \cH^{s+1}_{k-1}\oplus \cW^s_k
$$
where $$\cW^s_k=((k-2+m-s)\;\xx\wedge\xx\bullet -(k-2+s)\;\xx\bullet\xx\wed )\cH^s_{k-2}$$
for $1\leq s\leq m-1$ and $k\geq 2,$ and $\cW^s_k=\{0\}$ otherwise.
Here $\cH^{s'}_{k'}=\{0\}$ unless $0\leq s'\leq m$ and $k'\in\bN_0.$
\end{thm}

In Section \ref{sinfra}, we obtain analogous decompositions for inframonogenic polynomials, see Theorem \ref{tdecompi} below.


\section{Special Monogenic Polynomials}

In this section, we study special polynomial solutions of the Dirac equation. In what follows,
special solutions are just those taking values in a~given subspace $V$ of the Clifford algebra $\bR_{0,m}.$
Since we consider the $H$-action on the space $\cP^*$ it is natural to assume that the subspace $V$ is invariant under the both side action of the Pin group $Pin(m),$ that is, $rVr^{-1}\subset V$ for each $r\in Pin(m).$
Obviously, in this case, for some $S\subset\{0,\ldots,m\},$ we have that
$$V=\bR_{0,m}^S\text{\ \ with\ \ }\bR_{0,m}^S=\bigoplus_{s\in S}\bR_{0,m}^s$$
because the spaces $\bR^s_{0,m}$ of $s$-vectors are all irreducible and mutually inequivalent with respect to the given action.
Denote by $\cP^S_k$ the space of $\bR^S_{0,m}$-valued polynomials $P$ in $\bR^m$ which are homogeneous of degree $k.$
In general, for $\cV\subset\cP^*$ put
$\cV^S_k=\cV\cap\cP^S_k$ as usual.

Furthermore, let us denote by, respectively, $\cM$ and $\tilde\cM$ the spaces of left and right monogenic polynomials, that is,
$$\cM=\{P\in\cP^*: \pa P=0\}\text{\ \ \ and\ \ \ }\tilde\cM=\{P\in\cP^*: P\pa=0\}.$$
Recall that $\cH=\{P\in\cP^*: \pa^+ P=0,\ \pa^-P=0\}$ 
is the space of polynomial solutions of the Hodge-de Rham system.  
It is easy to see and well-known that the space $\cH$ is formed just by all two-sided monogenic polynomials, that is, $$\cH=\cM\cap\tilde\cM.$$
Indeed, for $P\in\cP^s_k,$ we have that $$P\pa=(-1)^s\tilde\pa P$$  where $\tilde\pa=\pa^+-\pa^-$ is the so-called modified Dirac operator.
Moreover, it is easy to see that
$$
\tilde\cM=\{P\in\cP^*: \tilde\pa P=0\}\text{\ \ \ and\ \ \ }
\cH_k=\bigoplus_{s=0}^m\cH_k^s.$$ See \cite{BDS} for details.

Given a~set $S\subset\{0,\ldots,m\},$ we are mainly interested in the space $\cM^S_k$ of $\bR^S_{0,m}$-valued left monogenic polynomials in $\bR^m$ which are homogeneous of degree $k.$ Let us give a~few known examples of such spaces.

\begin{ex}
For $S=\{s\},$ we have that $\cM^s_k=\tilde\cM^s_k=\cH^s_k.$
In other words, for $s$-vector valued functions, all three notions of monogeneity coincide.
As we shall see again below the spaces $\cH^s_k$ of homogeneous solutions of the Hodge-de Rham system (\ref{HdR}) are basic building blocks for the $H$-action.
\end{ex}

\begin{ex}
For $S=\{0,\ldots,m\},$ we have that $\cM^S_k=\cM_k.$ In \cite{DLS2}, as an easy application of Theorem \ref{tFischerH},
the following multiplicity free irreducible decomposition of the space $\cM_k$ has been obtained.
\end{ex}

\begin{thm}\label{tmonog}
Under the $H$-action, the space $\cM_k$ decomposes into inequivalent irreducible pieces as
$$\cM_k=\left(\bigoplus_{s=0}^m \cH^s_k\right)\oplus\left(\bigoplus_{s=1}^{m-1}((k-1+m-s)\xx\bullet-(k-1+s)\xx\wed)\cH^s_{k-1}\right)$$
\end{thm}

\begin{rem} We rewrite now the result of Theorem \ref{tmonog} as in \cite{lav}.
Let us define the Euler operator $E$ and fermionic Euler operators $\pa^+\rfloor$ and $\pa^-\rceil$ by
\begin{equation}\label{eulers}
E=\sum_{j=1}^mx_j\pa_{x_j},\ \ \ \pa^+\rfloor=-\sum_{j=1}^m e_j\wedge e_j\bullet\text{\ \ \ and\ \ \ }
\pa^-\rceil=-\sum_{j=1}^m e_j\bullet e_j\wedge.
\end{equation}
For $P\in\cP^s_k,$ it is easy to see that $$EP=kP,\ \ \ \pa^+\rfloor P=sP\text{\ \ and\ \ }\pa^-\rceil P=(m-s)P.$$
See \cite{BDS} for details.
Putting $A=E+\pa^+\rfloor$ and $B=E+\pa^-\rceil,$
Theorem \ref{tmonog} tells us that
\begin{equation}\label{lmonog}
\cM_k=\cH_k\oplus X\cH_{k-1}\text{\ \ \ with\ \ \ }X=x\wedge A-x\bullet B.
\end{equation}
Here $X\cH_{k-1}=\{XP:P\in\cH_{k-1}\}.$
Notice that $X\cH^0_{k-1}=\{0\}$ and $X\cH^m_{k-1}=\{0\}.$ Moreover, we can obtain easily
an analogous decomposition for right monogenic polynomials. Indeed, we have that
\begin{equation}\label{rmonog}
\tilde\cM_k=\cH_k\oplus \tilde X\cH_{k-1}\text{\ \ \ with\ \ \ }\tilde X=x\wedge A+x\bullet B.
\end{equation}
\end{rem}

\begin{ex} Assume that  $r,p$ and $q$ are non-negative integers such that $p<q$\\ and $r+2q\leq m.$
Putting $S=\{r+2p,r+2p+2,\ldots,r+2q\},$ we call $$\pa f=0\text{ for }\bR_{0,m}^S\text{-valued functions }f$$ the generalized Moisil-Th\'eodoresco  system of type $(r,p,q).$
In \cite{DLS}, the space $\cM^S_k$ of $k$-homogeneous solutions of this system is decomposed into a~direct sum of pieces isomorphic to spaces $\cH^s_k$ of homogeneous solutions of the Hodge-de Rham system.
In Corollary \ref{tsmonog} below, we describe more explicitly pieces of the decomposition of the space $\cM^S_k$ even in the general case.
\end{ex}

Since $\cM^S_k$ is always an invariant subspace of the space $\cM_k$ the next result follows easily from Theorem \ref{tmonog}.
\begin{cor}\label{tsmonog} Let $S\subset\{0,\ldots,m\}$ and $S'=\{s:s\pm 1\in S\}.$
Under the $H$-action, the space $\cM^S_k$ decomposes into inequivalent irreducible pieces as
$$\cM_k^S=\left(\bigoplus_{s\in S} \cH^s_k\right)\oplus\left(\bigoplus_{s\in S'}((k-1+m-s)\xx\bullet-(k-1+s)\xx\wed)\cH^s_{k-1}\right).$$
In particular, we have that
$\cM^S_k=\cH^S_k\oplus X\cH^{S'}_{k-1}$
where $X$ is as in (\ref{lmonog}).
\end{cor}


\section{Inframonogenic Polynomials}\label{sinfra}

In \cite{MPS1,MPS2}, inframonogenic functions have been recently introduced and studied. In particular, in \cite{MPS1}, an analogue of the Fischer decomposition for inframonogenic polynomials is given. Recall that an $\bR_{0,m}$-valued polynomial $P$ in $\bR^m$ is said to be inframonogenic if $$\pa P\pa=0.$$
Defining the modified Laplace operator $\tilde\Delta$ by
$$\tilde\Delta=-(\pa^+\pa^--\pa^-\pa^+),$$ it is well-known that a~polynomial $P$ is inframonogenic if and only if $\tilde\Delta P=0,$ see \cite{BDS}. Indeed, for $P\in\cP^s_k,$ we have that
$$\pa P\pa=(-1)^s\tilde\Delta P$$ and the operator $\tilde\Delta$ is scalar in the sense that it preserves the order of the multiplicative function on which it acts. Moreover, let us remark that,
for $P\in\cP^s_k,$ we have that
$$\xx P\xx=(-1)^s(\xx\bullet\xx\wedge-\;\xx\wedge\xx\bullet)P.$$ See \cite{BDS} for details.
In \cite{MPS1}, the next decomposition has been obtained (cf. (\ref{hFischer})):
\begin{equation}\label{iFischer}
\cP^*=\bigoplus_{s=0}^{m}\bigoplus_{k=0}^{\infty}\bigoplus_{p=0}^{\infty}\xx^p(\Ker^s_k\tilde\Delta)\xx^p
\end{equation}
where $\Ker_k^s\mD=\{P\in\cP_k^s: \mD P=0\}.$ In Theorem \ref{tdecompi} below, we present decompositions of homogeneous inframonogenic polynomials in terms of two-sided monogenic ones which are quite analogous to the decompositions of  homogeneous harmonic polynomials given in Theorem \ref{hom}. 
Moreover, decompositions of homogeneous polynomials which are harmonic and inframonogenic at the same time are given.

\begin{thm}\label{tdecompi} Let $0\leq s\leq m$ and $k\in\bN_0.$\\
(i) Under the $H$-action, $\Ker_k^s\mD$ decomposes into inequivalent irreducible pieces as
$$
\Ker_k^s\mD=\cH^s_k\oplus\xx\wedge \cH^{s-1}_{k-1}\oplus\xx\bullet \cH^{s+1}_{k-1}\oplus \tilde \cW^s_k
$$
where, for $1\leq s\leq m-1$ and $k\geq 2,$ we have that  $$\tilde \cW^s_k=((c_1+1)c_2\;\xx\wedge\xx\bullet\;+(c_2+1)c_1\;\xx\bullet\xx\wed)\cH^s_{k-2}$$
with $c_1=k-2+s$ and $c_2=k-2+m-s,$ and $\tilde\cW^s_k=\{0\}$ otherwise.\\
Here $\cH^{s'}_{k'}=\{0\}$ unless $0\leq s'\leq m$ and $k'\in\bN_0.$\\
(ii) Moreover, we have that
$$
\Ker^s_k\Delta\cap\Ker_k^s\mD=\cH^s_k\oplus\xx\wedge \cH^{s-1}_{k-1}\oplus\xx\bullet \cH^{s+1}_{k-1}.
$$
\end{thm}

\begin{proof}
(a) For a~proof of the statement (ii), see \cite[proof of Theorem 1]{lav}. 

\medskip\noindent
(b) For $1\leq s\leq m-1$ and $k\geq 2,$ we show now that
$$\tilde \cW^s_k=\Ker^s_k\tilde\Delta\cap(\xx\wedge\xx\bullet \cH^s_{k-2}\oplus \xx\bullet\xx\wedge \cH^s_{k-2}).$$
To do this, let
a~polynomial $P$ belong to the space $$\xx\wedge\xx\bullet \cH^s_{k-2}\oplus \xx\bullet\xx\wedge \cH^s_{k-2}.$$
Then there are uniquely determined polynomials $P_1,P_2\in \cH^s_{k-2}$ such that
$$P=\xx\wedge\xx\bullet P_1+\xx\bullet\xx\wedge P_2.$$
Obviously, it remains to show that $\tilde\Delta P=0$ if and only if the polynomial $P$ lies in $\tilde\cW^s_k.$
To show this we use the following well-known relations (see e.g. \cite{BDS}):

\begin{lem}\label{lrels}
If we put $\{T,S\}=TS+ST$ for linear operators $T$ and $S$ on the space $\cP^*,$ then we have that
\begin{equation*}
\begin{array}{lll}
\{\xx\wed,\xx\wed\}=0, &\{\xx\bul,\xx\bul\}=0, &\{\xx\wed,\xx\bul\}=-\sum_{j=1}^m x_j^2=-|\xx|^2,\medskip\\{}
\{\pa^+,\pa^+\}=0, &\{\pa^-,\pa^-\}=0, &\{\pa^+,\pa^-\}=-\Delta,\medskip\\{}
\{\xx\bul,\pa^+\}=-A, &\{\xx\wed,\pa^-\}=-B, &\{\xx\bul,\pa^-\}=0=\{\xx\wed,\pa^+\}.
\end{array}
\end{equation*}
\end{lem}

\noindent
Obviously, Lemma \ref{lrels} gives us that $\mD P=-2c_1(c_2+1)P_1+2(c_1+1)c_2P_2$ because
$$\mD\;\xx\wedge \xx\bullet P_1=-2c_1(c_2+1)P_1
\text{\ \ and \ \ }\mD\;\xx\bullet\xx\wedge P_2=2(c_1+1)c_2P_2$$
with $c_1=k-2+s$ and $c_2=k-2+m-s.$
Whence we conclude that $\mD P=0$ if nad only if
$$P_2=\frac{c_1(c_2+1)}{c_2(c_1+1)}P_1,$$
which finishes the proof.

\medskip\noindent
(c) We prove the statement (i). Assume that $k\geq 2$ and $1\leq s\leq m-1.$ Otherwise, we can argue in an analogous way.
Denoting $$N^s_k=\cH^s_k\oplus\xx\wedge \cH^{s-1}_{k-1}\oplus\xx\bullet \cH^{s+1}_{k-1}\oplus \tilde \cW^s_k,$$ it is easy to see, by (ii) and (b), that $N^s_k\subset\Ker^s_k\tilde\Delta.$
Moreover, by (\ref{iFischer}), we have that
$$\cP^s_k=\bigoplus^{[k/2]}_{p=0}\xx^p(\Ker^s_{k-2p}\tilde\Delta)\xx^p.$$
Hence to prove the opposite inclusion $N^s_k\supset\Ker^s_k\tilde\Delta$ it is sufficient to show that
\begin{equation}\label{psk}
\cP^s_k=\bigoplus^{[k/2]}_{p=0}\xx^p N^s_{k-2p}\; \xx^p.
\end{equation}
To do this notice that, putting $w_1=\xx\wedge\xx\bul$ and $w_2=\xx\bullet\xx\wed,$ we get
$$\xx^p(\xx\wedge \cH^{s-1}_{k-2p-1})\xx^p=w_1^p\;\xx\wedge\cH^{s-1}_{k-2p-1}\text{\ \ and\ \ }
\xx^p(\xx\bullet\cH^{s+1}_{k-2p-1})\xx^p=w_2^p\;\xx\bullet\cH^{s+1}_{k-2p-1}.$$
Moreover, since $\xx\;\cH^s_{k-2}\;\xx\oplus\tilde\cW^s_k=
w_1\cH^s_{k-2}\oplus w_2\cH^s_{k-2}$
we have that, for $p\geq 1,$
$$\xx^p\;\cH^s_{k-2p}\;\xx^p\oplus\xx^{p-1}\;\tilde\cW^s_{k-2p+2}\;\xx^{p-1}=
w_1^p\cH^s_{k-2p}\oplus w_2^p\cH^s_{k-2p}.$$
Finally, Theorem \ref{tFischerH} gives us that
$$
\cP^s_k=\cH^s_k\oplus\bigoplus^{[k/2]}_{p=1}(w_1^p\cH^s_{k-2p}\oplus w_2^p\cH^s_{k-2p})
\oplus\bigoplus^{[(k-1)/2]}_{p=0}(w_1^p\xx\wedge\cH^{s-1}_{k-2p-1}\oplus w_2^p\xx\bullet\cH^{s+1}_{k-2p-1}).
$$
Using these observations, we get easily the decomposition (\ref{psk}), which completes the proof of (i).
\end{proof}


\subsection*{Acknowledgment}

I am grateful to R. Delanghe and V. Sou\v cek for useful conversations.


\bigskip\bigskip\bigskip

\noindent
Roman L\'avi\v cka,\\ Mathematical Institute, Charles University,\\ Sokolovsk\'a 83, 186 75 Praha 8, Czech Republic\\
email: \texttt{lavicka@karlin.mff.cuni.cz}

\end{document}